\documentclass[11pt]{amsart}
\usepackage{amssymb, amsfonts, amsbsy, latexsym, epsfig}

\usepackage{amsthm}
\usepackage{amsmath}
\usepackage{color}

\newtheorem{thm}{Theorem}[section]
\newtheorem{theorem}[thm]{Theorem}

\newtheorem{corollary}[thm]{Corollary}
\newtheorem{problem}[thm]{Problem}
\newtheorem{notation}[thm]{Notation}
\newtheorem{example}[thm]{Example}

\DeclareMathOperator{\lspan}{span}

\newtheorem{lemma}[thm]{Lemma}
\newtheorem{proposition}[thm]{Proposition}

\newtheorem{defn}[thm]{Definition}
\newtheorem{definition}[thm]{Definition}
\theoremstyle{remark}
\newtheorem{remark}[thm]{Remark}

\newcommand{\RR}{\mathbb R}
\newcommand{\NN}{\mathbb N}
\newcommand{\CC}{\mathbb C}
\newcommand{\HH}{\mathbb H}

\begin{document}

\title{Phase Retrieval in $\ell_2(\RR)$}
\author[Andrade, Casazza, Cheng, Haas, Tran
 ]{Sara Botelho-Andrade, Peter G. Casazza, Desai Cheng, John Haas,
 and Tin T. Tran}
\address{Department of Mathematics, University
of Missouri, Columbia, MO 65211-4100}

\thanks{The authors were supported by
 NSF DMS 1609760; NSF DMS 1725455; NSF ATD 1321779; and ARO  W911NF-16-1-0008 }

\email{sandrade102087@gmail.com, Casazzap@missouri.edu} \email{chengdesai@yahoo.com, terraformthedreamscape@gmail.com}
\email{tinmizzou@gmail.com}

\begin{abstract}
We will review the major results in finite dimensional real
phase retrieval for vectors and projections.  We then 
\begin{enumerate}
\item prove that many of these theorems hold in infinite dimensions,
\item give counter-examples
to show that many others fail in infinite dimensions, 
\item  list finite dimensional
results are unknown for $\ell_2$.
\end{enumerate}
\end{abstract}

\maketitle

\section{introduction}

Phase retrieval is one of the most applied and studied areas of
research today.  Phase retrieval for Hilbert space frames was
introduced in \cite{BCE} and quickly became an industry.  Although
much work has been done on the complex infinite dimensional case
of phase retrieval, only one paper exists on infinite dimenional
real phase retrieval \cite{CCD}.  Here we will review the major
results on finite dimensional real phase retrieval and show:
\begin{enumerate}
\item  Which results hold in infinite dimensions;
\item  Which results fail in infinite dimensions;
\item  Which results are unknown in infinite dimensions.
\end{enumerate}

We will need the definition of a Hilbert space frame.

\begin{definition}
A family of vectors $\{x_i\}_{i\in I}$ in a finite or infinite
dimensional Hilbert space $\HH$ is a {\bf frame} if there are
constants $0<A\le B < \infty$ so that
\[ A\|x\|^2 \le \sum_{i\in I}|\langle x,x_i\rangle|^2 
\le B\|x\|^2,\mbox{ for all }x \in \HH.\]
\end{definition}

\begin{enumerate}
\item If $A=B$ this is an A-{\bf tight} frame.
\item If $A=B=1$, this is a Parseval frame.
\end{enumerate}

We also need to work with Riesz sequences.

\begin{definition}
A family $X=\{x_i\}_{i\in I}$ in a finite or infinite dimensional
Hilbert space $\HH$ is a {\bf Riesz sequence} if there are
constants $0<A\le B<\infty$ satisfying for all sequences
of scalars $\{a_i\}_{i\in I}$ we have:
\[ A \sum_{i\in I}|a_i|^2 \le \|\sum_{i\in I}a_ix_i\|
\le B \sum_{i\in I}|a_i|^2.\]
If the closed linear span of $X$ equals $\HH$, we call $X$
a {\bf Riesz basis}.
\end{definition}

The complement property and full spark will be a major tool here.

\begin{definition}
A family of vectors $\{x_i\}_{i=1}^{\infty}$ 
in $\ell_2$ has the {\bf complement
property} if for every $I\subset \NN$ either $\overline{\lspan}_{i\in I}= \ell_2$ or $\overline{\lspan}_{i\in I^c}= \ell_2$.
\end{definition}

\begin{definition}
A family of vectors $\{x_i\}_{i=1}^m$ in $\RR^n$ is {\bf full
spark} if for every $I\subset [m]$ with $|I|=n$ we have
that $\{x_i\}_{i\in I}$ is linearly independent (hence spans
$\RR^n$).
\end{definition}

Throughout the paper,  $\{e_i\}_{i=1}^{\infty}$ will be used to denote the canonical orthonormal basis for the real Hilbert space $\ell_2$.

\section{Finite Dimensional Results Which Carry Over to Infinite
Dimensions}

In this section we look at finite dimensional real phase retrieval
and norm retrieval results which carry over to infinite dimensions.

\subsection{Phase Retrieval}

We start with the definitions.

\begin{definition}
A family of vectors $\{x_i\}_{i\in I}$ 
in a Hilbert space 
$\HH$ does {\bf phase retrieval} if whenever $x,y\in \HH$ satisfy
\[ |\langle x,x_i\rangle|=|\langle y,x_i\rangle|, \mbox{ for all }
i=1,2,\ldots,\]
then $x=\pm y$.

A family of projections $\{P_i\}_{i\in I}$ on a Hilbert space
$\HH$ does {\bf phase retrieval} if whenever $x,y\in \HH$
and
\[ \|P_ix\|=\|P_iy\|,\mbox{ for all }i=1,2,\ldots,\]
then $x=\pm y$.
\end{definition}

The following result appeared in \cite{CCD}.  The 
corresponding finite dimensional result first appeared in
\cite{CCJ}.

\begin{theorem}
A family of vectors in real $\ell_2$ does phase retrieval if and
only if it has the complement property.
\end{theorem}

It follows that results on phase retrieval
in the finite dimensional case which just depend on the complement property will
also hold in $\ell_2$. 

\begin{theorem}\label{pr1}
Let $X=\{x_i\}_{i=1}^{\infty}$ do phase retrieval.
\begin{enumerate}
\item Then
 so does
$\{a_ix_i\}_{i=1}^{\infty}$ for $a_i\not= 0$ for all i.
\item If $T$ is an invertible operator then $\{Tx_i\}_{i=1}^{\infty}$ does phase retrieval.
\end{enumerate}
\end{theorem}

The finite dimensional version of the next theorem first appeared
in \cite{E}.

\begin{theorem}\label{Dan}
A family of projections $\{P_i\}_{i=1}^{\infty}$ 
on $\ell_2$ does phase
retrieval in $\ell_2$ if and only if for every $0\not= x \in \ell_2$, $\overline{\lspan}\{P_ix\}_{i=1}^{\infty}=\ell_2.$ 
\end{theorem}

\begin{proof}
$(\Rightarrow)$  We proceed by way of contradiction.  
So assume that there is an $0\not= x \in \ell_2$
and $\{P_ix\}_{i=1}^{\infty}$ does not span $\ell_2$.  Choose
$0\not= y \in \ell_2$ so that $y\perp P_ix$ for all $i=1,2,\ldots$.
Let $u=x+y$ and $v=x-y$.  Then since $P_iy \perp P_ix$ for all
$i$, we have that
\[ \|P_i(x+y)\|^2=\|P_ix\|^2+\|P_iy\|^2= \|P_i(x-y)\|^2.\]
If $\{P_i\}_{i=1}^\infty$ does phase retrieval, then 
\[ x+y=\pm (x-y).\]
This implies $x=0$ or $y=0$, which is a contradiction.

$(\Leftarrow)$  The proof of Edidin's theorem in \cite{CC} works
directly here.  
\end{proof}

So results in finite dimensions which depend only on Edidin's 
theorem also hold in $\ell_2$.

\begin{corollary}
The following are equivalent for a 
family of projections $\{P_i\}_{i=1}^{\infty}$ on $\ell_2$:
\begin{enumerate}
	\item $\{P_i\}_{i=1}^{\infty}$ fails phase retrieval.
	\item There are vectors $\|x\|=\|y\|=1$ in $\ell_2$ so that
	$P_ix \perp P_iy$ for all $i$.
\end{enumerate} 
\end{corollary}

The finite dimensional version of the following theorem appeared in
\cite{CCJ}.  The proof in that case immediately works in
$\ell_2$.

\begin{theorem}
Let $\{P_i\}_{i=1}^\infty$ do phase retrieval on $\ell_2$.
\item Then $\{(I-P_i)\}_{i=1}^\infty$ does phase retrieval if and only
if it does norm retrieval.
\end{theorem}

In the finite dimensional case it is known \cite{CCJ} that there
are projections $\{P_i\}_{i=1}^m$ which do phase retrieval
but $\{(I-P_i)\}_{i=1}^m$ fails phase retrieval.  
We now give the infinite
dimensional version of this

\begin{theorem}There is a set of vectors in $\ell_2$ which do phase retrieval but their perps fail to do phase retrieval.
\end{theorem} 
\begin{proof}
	First, we will show that for each $n\in \mathbb{N}$, there exista a full spark set of $(2n-1)$ vectors $\{x_{nk}\}_{k=1}^{2n-1}$ in $\mathbb{R}^n$ such that the first coordinate of the vector $(I-P_{nk})(\varphi_n)$ equals zero, where
	$$\varphi_n=\left(1, \dfrac{1}{2}, \ldots, \dfrac{1}{n}\right)$$ and $P_{nk}$ is the orthogonal projection onto $x_{nk}$.
	
	Define the following subset of $\RR^n$:
	$$A_n:=\left\{\left(\sum_{i=1}^{n-2}t^{2i}+\left(n-n\sum_{i=2}^{n-1}\dfrac{t^{i-1}}{i}\right)^2, t, t^2, \ldots t^{n-2}, n-n\sum_{i=2}^{n-1}\dfrac{t^{i-1}}{i} \right) : t\in \RR\right\}.$$
	
	Let any $x\in A_n$, and denote by $P_x$ the projection onto $x$.  Then
	$$(I-P_{x})(\varphi_n)=\varphi_n-\langle \varphi_n,\dfrac{x}{\Vert x\Vert}\rangle \dfrac{x}{\Vert x\Vert}.$$
	
	Denote by $a_1$ and $b_1$ the first coordinate of $x$ and $(I-P_x)(\varphi_n)$ respectively.  Then we have
	$$a_1=\Vert x\Vert^2-a_1^2,$$ and hence
	$$	b_1=1-\dfrac{1}{\Vert x\Vert^2}(a_1+1)a_1=0.$$
	
	Now we will show that for any finite family of hyperplanes in $\mathbb{R}^n$, there exists a point in $A_n$ that does not lie in any of these hyperplanes, and therefore there exists a full spark of $(2n-1)$ vectors $\{x_{nk}\}_{k=1}^{2n-1} \subset A_n$.
	
	Indeed, let $\{W_i\}_{i=1}^k$ be any finite 
	set of hyperplanes in $\RR^n$. 
	Suppose, by way of contradiction, that $B_n\subset \cup_{i=1}^kW_i$. Then there exists $j\in\{1, \ldots, k\}$ such that $W_j$ contains infinitely many vectors in $A_n$.
	
	Let $u=(u_1, u_2, \ldots, u_n)\in W_j^{\perp}, u\not=0$. Then we have
	
	$$\langle u, x_t\rangle =0$$
	for infinitely many $x_t\in A_n$.
	
	Hence,
	$$u_1\left(\sum_{i=1}^{n-2}t^{2i}+\left(n-n\sum_{i=2}^{n-1}\dfrac{t^{i-1}}{i}\right)^2\right)+\sum_{i=1}^{n-2}u_{i+1}t^i+u_n\left(n-n\sum_{i=2}^{n-1}\dfrac{t^{i-1}}{i}\right)=0$$ for infinitely many $t\in \RR$.
	
	This implies $u_1=u_2=\cdots=u_n=0$, which is impossible.
	
	Thus, we have shown that for each $n$, there exists a full spark set of $(2n-1)$  vectors $\{x_{k,n}\}_{k=1}^{2n-1}$ in $\RR^n$ such that the first coordinate of the vector $(I-P_{nk})(\varphi_n)$ equals zero. Notice that   $\{x_{nk}\}_{k=1}^{2n-1}$ does phase retrieval in $\RR^n.$
	
	Now, for each $ n$, we consider $x_{nk}$ as a vector in $\ell_2$, where its $j$-coordinate is zero when $j>n$. Then the collection of all $x_{nk}$, $n\in \mathbb{N}, k= 1, 2,\ldots, 2n-1, $ do phase retrieval in $\ell_2$.
	Indeed, suppose that $\vert \langle x, x_{nk}\rangle\vert =\vert\langle y, x_{nk}\rangle \vert$ for all $n, k$ but $x\not=\pm y$. Then there is a $n_0$ such that $$(x(1), x(2), \ldots ,x(n_0))\not=\pm (y(1), y(2), \ldots, y(n_0)).$$ But then the corresponding $\{x_{n_0k}\}_{k=1}^{2n_0-1}$ does phase retrieval $\in \mathbb{R}^{n_0}$ which implies $(x(1), x(2), \ldots ,x(n_0))=\pm (y(1), y(2), \ldots, y(n_0))$, a contradiction.
	
	Finally, we show that $\{x_{nk}^\perp\}_{n=1, k=1}^{\infty,\ \ 2n-1}$ fails phase retrieval in $\ell_2$. We will use the same notation $P_{nk}$ for the projection onto $x_{nk}\in \ell_2$, and 
	let 
	$\varphi=\left\{\frac{1}{n}\right\}_{n=1}^{\infty}\in \ell_2.$ Then by our construction, the first coordinate of $(I-P_{nk})(\varphi)$ equals zero for all $k, n$. Therefore, $\overline{\lspan}\{(I-P_{x_{nk}})(\varphi)\}_{nk}\not=\ell_2$. Therefore $\{x_{nk}^\perp\}_{n=1,k=1}^{\infty, \ \ 2n-1}$ fails phase retrieval by Theorem \ref{Dan}.

\end{proof}

In finite dimensions it is known
\cite{BCE} that any family of vectors doing phase retrieval
must contain at least $(2n-1)$-vectors.
It follows that a full spark family of vectors 
$\{x_i\}_{i=1}^{2n-1}$ does phase retrieval (since it has
complement property) but if we delete
any vector it fails phase retrieval. Now we give a construction to show
that this result holds in infinite dimensions.  
The following example shows that there is a family of vectors in $\ell_2$ which does phase retrieval but we cannot drop any vector and maintain phase retrieval.
This also 
contains a new construction for frames doing phase retrieval.

First, we need the following lemma:
\begin{lemma}\label{Lem1}
	Let $\{e_i\}_{i=1}^\infty$ be the canonical orthonormal basis for $\ell_2$. For any fixed $i$, if $x$ is orthogonal to $e_i+e_j$ for infinitely many $j>i$, then $\langle x, e_i\rangle=\langle x, e_j\rangle=0$ for all such $j$.
\end{lemma}
\begin{proof}
	Let $K=\{j: j>i, \langle x, e_i+e_j\rangle=0\}$, then by assumption, the cardinality of $K$ is infinite. 
	
	It is clear that $|\langle x, e_i\rangle|=|\langle x, e_j\rangle|$ for all $j\in K$. Suppose by a contradition that $| \langle x, e_j\rangle|>0$ for all $j\in K$. Then we have
	\[\Vert x\Vert^2\geq \sum_{j\in K}|\langle x, e_j\rangle|^2=\infty,\] a contradiction.
\end{proof}
\begin{example} Let the family of vectors $X=\{e_i+e_j\}_{i<j}$. Then $X$ does phase retrieval in $\ell_2$ but we cannot drop any vector of $X$ and maintain phase retrieval. 
\end{example}

\begin{proof}
	Let $I$ be any subset of the set $\{(i, j): i<j\}$, and we can assume that $(1,j)\in I$ for infinitely many $j$. We will show that either $\{e_i+e_j\}_{(i, j)\in I}$ or $\{e_i+e_j\}_{(i, j)\in I^c}$ spans $\ell_2$.
	Suppose $\{e_i+e_j\}_{(i, j)\in I}$ does not span $\ell_2$. We will show that $\{e_i+e_j\}_{(i, j)\in I^c}$ spans $\ell_2$. 
	
	Let any $x=(x(1), x(2), \ldots)$ be such that $\langle x, e_i+e_j\rangle=0$ for all $(i, j)\in I^c$.
	
	By assumption, there is $y=(y(1), y(2), \ldots), y\not=0$ and $\langle y, e_i+e_j\rangle=0$ for all $(i, j)\in I$.		 	
	Let $s$ be the smallest number such that $y(s)\not=0$. By Lemma \ref{Lem1}, $(s, j)\notin I$ for infinitely many $j>s$. Hence there is $t>s$ such that $(s, j)\in I^c$ for all $j\geq t$. Again, by Lemma \ref{Lem1}, we get
	$$x(s)=x(j)=0 \mbox{ for all } j\geq t.$$
	
	We will now show that $x(j)=0$ for all $j=1,2\ldots t-1$.
	Suppose there is $1\leq j< s$ such that $x(j)\not=0$. This implies $(j, s) \notin I^c$. Thus $(j, s)\in I$ and hence $y(j)\not=0$. But this contradicts the way we chose $s$. So $x(j)=0$ for all $1\leq j<s$.
	
	Now let any $s<j<t$. 
	If $(s,j)\in I^c$, then $x(j)=0$. 
	If $(s,j)\in I$, then $y(j)\not=0$. Note that by assumption, $(1, j) \in I$  for infinitely many $j$, and hence by Lemma \ref{Lem1}, we get that $y(1)=0$. Thus, $(1, j)\notin I$. Therefore $(1, j)\in I^c$ and so $x(j)=x(1)=0$.
	This completes the proof that $\{e_i+e_j\}_{(i, j)\in I^c}$ span $\ell_2$.
	
	Now we will show that we cannot drop any vector of $X$ and maintain phase retrieval.
	
	Fix any $(k, \ell), k<\ell$. Consider $Y=\{e_i+e_j : i<j, (i, j)\not=(k, \ell)\}$.
	Let $$x=e_k+e_{\ell}, \ y=e_k-e_{\ell}.$$
	Clearly, $x\not= \pm y$. For any vector $e_i+e_j\in Y$, we compute:
	$$\langle x, e_i+e_j\rangle= \langle e_k, e_i\rangle+\langle e_k, e_j\rangle+\langle e_\ell, e_i\rangle+\langle e_\ell, e_j\rangle,$$
	$$\langle y, e_i+e_j\rangle= \langle e_k, e_i\rangle+\langle e_k, e_j\rangle-\langle e_\ell, e_i\rangle-\langle e_\ell, e_j\rangle.$$
	If $i=k$ then $j\not=\ell$, $i<\ell$ and $k<j$. Thus
	$$\langle x, e_i+e_j\rangle =\langle y, e_i+e_j\rangle=1.$$
	If $j=k$, then $i<j=k<\ell$. So
	$$\langle x, e_i+e_j\rangle =\langle y, e_i+e_j\rangle=1.$$
	
	Consider the case $i, j\not=k$. 
	If $i=\ell$ then $j\not=\ell$. Hence 
	$$\langle x, e_i+e_j\rangle=1, \mbox { and } \langle y, e_i+e_j\rangle=-1.$$
	If $i\not=\ell$ and $j=\ell$ then 
	$$\langle x, e_i+e_j\rangle=1, \mbox { and } \langle y, e_i+e_j\rangle=-1.$$
	Finally, if $i\not=\ell$ and $j\not=\ell$ then 
	$$\langle x, e_i+e_j\rangle=\langle y, e_i+e_j\rangle=0.$$
	Thus, in all cases, we always have that
	$$|\langle x, e_i+e_j\rangle|=|\langle y, e_i+e_j\rangle|, \mbox { for all } e_i+e_j\in Y.$$ Since $x\not=\pm y$, $Y$ cannot do phase retrieval.
	
\end{proof}

\subsection{Full Spark}

\begin{remark}
It is known that the full spark families of vectors in $\RR^n$
are dense in $\RR^n$ in the sense that given any family of
vectors $x=\{x_i\}_{i=1}^m$ in $\RR^n$ and any $\epsilon >0$, there
is a full spark family of vectors $Y=\{y_i\}_{i=1}^m$ in $\RR^n$
so that
\[ d(x,y)^2 =\sum_{i=1}^m\|x_i-y_i\|^2 < \epsilon.\]
\end{remark}

 One interpretation of the definition of full spark is that any
minimal number of vectors in the set which could possibly span, must span,  
i.e. any subset of $n$-vectors must span.  The corresponding statement
for $\ell_2$ is:

\begin{definition}
A family of vectors $\{x_i\}_{i=1}^{\infty}$ is {\bf full spark} in $\ell_2$ if every infinite 
subset spans $\ell_2$.
\end{definition}

A full spark set clearly has complement property and hence does phase retrieval in the infinite dimensional
case. 

\begin{theorem}
There exist full spark families of vectors in $\ell_2$ which
then do phase retrieval.
\end{theorem}

\begin{proof}
Such an example  can be found in Theorem 2 of \cite{V}.

There is another simpler way to do this using the following argument. Instead of $\ell_2$  consider $L_2[0, 1]$. It is known that if a sequence $a_n \neq a$ of  numbers (real or complex) tends to $a$ when $n \to \infty$, then the sequence of functions $f_n(t) = e^{a_n t}$ spans $L_2[0, 1]$ (this is a standard application of the Hahn-Banach theorem together with the uniqueness theorem for holomorphic functions, see more in Appendix III of \cite{L}.) Since every subsequence of $a_n$ also has the same limit, every subsequence of $f_n$ also spans $L_2[0, 1]$.
\end{proof}

\subsection{Norm Retrieval}

\begin{definition}
A family of projections $\{P_i\}_{i\in I}$ 
($I$ is finite or infinite) does {\bf norm 
retrieval} in $\HH$ if for any $x,y\in \HH$ we have:
\[ \|P_ix\|=\|P_iy\|\mbox{ for all }i\in I
\mbox{ then }\|x\|=\|y\|.\]
\end{definition}

The finite dimensional version of the next theorem first
appeared in \cite{CC}.

\begin{theorem}
A family of projections $\{P_i\}_{i=1}^{\infty}$ on $\ell_2$
does norm retrieval if and only if for any $x\in \ell_2$, $x\in \overline{\lspan}\{P_ix\}_{i=1}^\infty$.
\end{theorem}
\begin{proof}
		$(\Rightarrow )$ Let any $x\in \ell_2$ and  $y\perp \overline{\lspan}\{P_ix\}_{i=1}^\infty$. 
		Then  	$$\langle P_ix, P_iy\rangle=\langle P_ix, y\rangle=0, \text{ for all } i.$$
		Let $u=x+y, v=x-y$ then we have $\Vert P_iu\Vert =\Vert P_iv\Vert$ for all $i$. 
		Hence $\Vert u\Vert=\Vert v\Vert$. Note that
$x=\frac{1}{2}(u+v)$ and $y=\frac{1}{2}(u-v)$.  Now we compute:
\[ 
\langle x,y\rangle = \frac{1}{4}\langle u+v,u-v\rangle
= \frac{1}{4}(\|u\|^2 - \|v\|^2)=0.\]
Thus, $\left(\overline{\lspan}\{P_ix\}_{i=1}^{\infty}\right)^\perp\subset x^\perp,$
		or equivalently, $x\in \overline{\lspan}\{P_ix\}_{i=1}^{\infty}.$
		
		$(\Leftarrow)$ Suppose $x, y \in \ell_2$, 
		and $\Vert P_ix\Vert =\Vert P_iy\Vert$ for all $i=1,2,\ldots$. 
		Set $$u=x+y, v=x-y.$$ Then we have $\langle P_iu, P_iv\rangle=0$ for all $i$. Hence $u\perp \overline{\lspan}\{P_iv\}_{i=1}^{\infty}$.
		Since $v\in \overline{\lspan}\{P_iv\}_{i=1}^{\infty}$ then $u\perp v$.
		It follows that $\Vert x\Vert =\Vert y\Vert.$
\end{proof}

The finite dimensional version of the next result first appeared
in \cite{CG}.

\begin{theorem}
A family of vectors $\{x_i\}_{i=1}^{\infty}$ does norm retrieval
in $\ell_2$ if and only if for every $I\subset \NN$ if $x\perp
\overline{\lspan}\{x_i\}_{i\in I}$ and $y\perp \overline{\lspan}\{x_i\}_{i\in I^c}$
then $x \perp y$.
\end{theorem}

\begin{proof}
		$(\Rightarrow )$ We may assume that $\|x\|=\|y\|=1$.
		Let $u=x+y, v=x-y$. Then $$| \langle u, x_i\rangle|=|\langle v, x_i\rangle|, \text { for all } i.$$
		Since $\{x_i\}_{i=1}^\infty$ does norm retrieval, $\Vert u\Vert=\Vert v\Vert.$ It follows that $x\perp y$.
		
		$(\Leftarrow)$ Suppose $| \langle x, x_i\rangle|=|\langle y, x_i\rangle|
		, \text { for all } i.$ 
		Denote $$I=\{i: \langle x, x_i\rangle =-\langle y, x_i\rangle\}.$$
		Then $$I^c=\{i: \langle x, x_i\rangle =\langle y, x_i\rangle\}.$$
		Let $u=x+y, v=x-y$.  Then $u\perp\overline{\lspan}\{x_i\}_{i\in I}$ and $v\perp \overline{\lspan}\{x_i\}_{i\in I^c}$.
		Therefore, $u\perp v$, and we get $\Vert x\Vert =\Vert y\Vert$.

\end{proof}

\begin{corollary}
If $\{x_i\}_{i=1}^{\infty}$ is a Parseval frame in $\ell_2$, let
for $I\subset \NN$, let 
\[ H_1= \overline{\lspan}\{x_i\}_{i\in I}\mbox{ and } H_2= \overline{\lspan}\{x_i\}_{i\in I^c},\]
then $H_1 \perp H_2$.
\end{corollary}

\section{Finite Dimensional Results Which Fail in Infinite
Dimensions}

It is known \cite{CCJ} that the families of vectors $\{x_i\}_{i=1}^m$ which do phase retrieval in $\RR^n$ are dense in 
the family of $m\ge(2n-1)$-element sets
of vectors in $\RR^n$.  This follows from the fact that full
spark families of $m\ge 2n-1$ vectors are dense and do phase
retrieval.  The corresponding result fails in infinite 
dimensions.

\begin{definition}
We say a family of sequences of vectors 
$\mathcal{F}$ is {\bf dense} in $\ell_2$
if given any sequence of vectors $Y=\{y_i\}_{i=1}^{\infty}\subset \ell_2$ and any $\epsilon >0$
there an $X=\{x_i\}_{i=1}^{\infty}\in \mathcal{F}$ so
that
\[ d(X,Y)^2 = \sum_{i=1}^{\infty}\|x_i-y_i\|^2 < \epsilon.\]
\end{definition}

\begin{remark}
Note that a Riesz basis cannot do phase retrieval since it clearly
fails complement property.
\end{remark}

\begin{proposition}\label{pp1}
Let $X=\{x_i\}_{i=1}^{\infty}\subset \ell_2$ be such that 
\[ \sum_{i=1}^{\infty}\|x_i-e_i\|^2 \le 1-\epsilon.\]
Then $X$ is a Riesz basis for $\ell_2$.
\end{proposition}

\begin{proof}
Define an operator $T:\ell_2 \rightarrow \ell_2$ by $Te_i=x_i$,
for all $i=1,2,\ldots$.  Given $a=\sum_{i=1}^{\infty}a_ie_i \in \ell_2$ we have
\begin{align*}
\|(I-T)a\|^2 &= \|\sum_{i=1}^{\infty}a_i(e_i-x_i)\|^2\\
&\le \sum_{i=1}^{\infty}|a_i|\|e_i-x_i\|\\
&\le \left ( \sum_{i=1}^{\infty}|a_i|^2\right ) \left (
\sum_{i=1}^{\infty}\|x_i-e_i\|^2 \right )\\
&\le (1-\epsilon)\|a\|^2.
\end{align*}
It follows that $T$ is an invertible operator and so $\{Te_i\}_{i=1}^{\infty}$ is a Riesz basis for $\ell_2$.
\end{proof}

\begin{proposition}
The families of vectors which do phase retrieval 
in $\ell_2$ are not dense
 in the infinite families of vectors in $\ell_2$.
\end{proposition}

\begin{proof}
Let $0<\epsilon <1$.  If $X=\{x_i\}_{i=1}^{\infty}$ is any family of
unit vectors with
\[ \sum_{i=1}^{\infty}\|x_i-y_i\|^2 < 1-\epsilon,\]
then $X$ is a Riesz basis and
hence cannot do phase retrieval.
\end{proof}

It is known in finite dimensions \cite{B,CCJ} that if $X=\{x_i\}_{i=1}^m$ does phase retrieval, there is an $\epsilon>0$ so that whenever
$Y=\{y_i\}_{i=1}^m$ satisfies:
\[ \sum_{i=1}^m\|x_i-y_i\|^2 < \epsilon,\]
then $Y$ does phase retrieval.  The above is called a
{\bf $\epsilon$-perturbation} of $X$. The corresponding result fails in
$\ell_2$ as was shown in \cite{CCD}.

\begin{theorem}
Given a frame $\{x_i\}_{i=1}^{\infty}$ doing phase retrieval in $\ell_2$ and an $\epsilon>0$, there is a frame $\{y_i\}_{i=1}^\infty$ which fails phase
retrieval in $\ell_2$ and satisfies:
\[ \sum_{i=1}^{\infty}\|x_i-y_i\|^2< \epsilon.\]
\end{theorem}

\begin{definition}
A set of vectors $\{x_i\}_{i=1}^{\infty}$ in $\ell_2$
 is
{\bf finitely full spark} if for every $I\subset \NN$ with
$|I|=n$, $\{P_Ix_i\}_{i=1}^{\infty}$ is full spark (i.e.
spark n+1), where $P_I$ is the orthogonal projection onto $\lspan\{e_i\}_{i\in I}$.
\end{definition}

We will have to generalize the definition of full spark.

\begin{definition}
A set of vectors $\{x_i\}_{i=1}^m$ in $\RR^n$ is full spark if
either they are independent or if $m\ge n+1$, then they have spark
$n+1$.
\end{definition}

\begin{proposition}
The finitely full spark families of vectors in $\ell_2$ are dense
in the infinite families of vectors in $\ell_2$.  In particular,
there are Riesz bases for $\ell_2$ which are finitely full spark,
and these families cannot do phase retrieval.
\end{proposition}

\begin{proof}
Let $\{y_i\}_{i=1}^{\infty}$ be a family of vectors in $\ell_2$
and fix $\epsilon >0$.
We will construct the vectors by induction.  To get started,
choose a vector $x_1$ with all non-zero coordinates
so that $\|x_1-y_1\|^2 < \frac{\epsilon}{2}$.
Now assume we have constructed vectors $\{x_i\}_{i=1}^m$ so
that for every $I\subset \NN$ with $|I|<\infty$, $\{P_Ix_i\}_{i=1}^m$ is full spark and $\|x_i-y_i\|^2 < \frac{\epsilon}{2^{i+1}}$.  For each finite subset $I\subset \mathbb{N}$, let
\[ \mathcal{G}_I=\bigcup\left\{\overline{\lspan}\left [ \{P_Ix_i\}_{i\in I'}\cup \{e_i\}_{i\in I^c}\right ]:I'\subset [m],\begin{cases}
 |I'|=m \mbox{ if } m+1\le |I|\\
 |I'|=|I|-1 \mbox{ if } |I|\le m
 \end{cases}\right\}.\]
 Let 
 $$\mathcal{F}=\bigcup_{n=1}^\infty\bigcup_{|I|=n}\mathcal{G}_I,$$
then $\mathcal{F}$ is a countable union of proper subspaces of 
$\ell_2$ and hence there exists a vector $y_{m+1}$ not
in $\mathcal{F}$ and $\|x_{m+1}-y_{m+1}\|^2 < \frac{\epsilon}{2^{m+1}}$. This provides the required family of finitely full
spark vectors. 
\end{proof}

\section{Lifting}

In this section we demonstrate an embedding of finite frames in higher dimensions such that the complement property is preserved, which we will refer to as ``lifting". We provide necessary and sufficient conditions for when such a construction is possible and an example to demonstrate problems that may arise in infinite dimensions. We begin with a few useful definitions.
\begin{defn} A frame $X = \{x_i\}_{i\in I}$ has the overcomplete complement property if for every $S\subset I$, either $\{x_i\}_{i\in S}$ or $\{x_i\}_{i\in S^c}$ spans and is linearly dependent, i.e. it is not a basis.
\end{defn}
The overcomplete complement property is a natural generalization of the usual complement property, as will be shown shortly. Next we specify exactly what types of embeddings we are considering.

\begin{defn} A frame $Y=\{y_i\}_{i=1}^m\subset \RR^{n+k}$ is a 
{\bf k-lifting} of a frame $\{x_i\}_{i=1}^m$ if 
\[ y_i|_{\RR^n}=x_i,\mbox{ for all }i=1, 2, \ldots, m.\]
\end{defn}

The next theorem classifies when 1-lifts are possible and provides a construction for the choice of coordinates to adjoin. 
\begin{theorem}\label{FiniteLift}A frame
$X=\{x_i\}_{i=1}^m\subset \RR^n$ can be $1$-lifted if and only if $X$ has the overcomplete complement property.
\end{theorem}
\begin{proof}
For the sufficiency we shall provide a constructive proof. The idea of the proof will be to produce a vector $v\in\RR^m$ such that the $i^{th}$ coordinate of $v$ will be the $(n+1)^{th}$ coordinate of $\hat{x}_i$. Given a subset $S\subset [m]$, by assumption either $X_S=\{x_i\}_{i\in S}$ or $X_{S^c}=\{x_i\}_{i\in S^c}$ spans $\RR^n$ and is linearly dependent. We begin by demonstrating an embedding of vectors from the spanning set that still span in $\RR^{n+1}$. Without loss of generality, in our notation we shall assume $X_S$ is always the overcomplete spanning set of vectors. Then for some choice of coefficients we have $\sum_{i\in S} \alpha_ix_i=0$ where $\alpha_i$ are not all zero. Denote $\alpha_S=(\alpha_1,\alpha_2,...,\alpha_{|S|})\in \RR^{|S|}$ and pick $\beta_S\in\RR^{|S|}$ such that $\langle \alpha_S,\beta_S\rangle\neq 0$. Define the embedded vectors $\hat{X}_S=\{\hat{x}_i\}_{i\in S}\in\RR^{n+1}$ as follows
\[ \hat{x}_i(j)=\begin{cases} 
      x_i(j) & j\in [n] \\
      \beta_S(i) & j=n+1. 
   \end{cases}\]
To show that $\hat{X}_S$ spans $\RR^{n+1}$, observe $\frac{1}{\langle \alpha_S,\beta_S\rangle}\sum_{i\in S}\alpha_i\hat{x}_i=e_{n+1}$. Since $X_S$ spans $\RR^n$, it follows that $\hat{X}_S$ spans $\RR^{n+1}$.\\
This construction gives a procedure for an embedding which spans the larger space $\RR^{n+1}$, but is dependent on the subset $S$. Also observe we haven't posed any conditions on how to extend the vectors in $S^c$. For each choice of $S$, we have the associated vectors $\alpha_S,\beta_S\in\RR^{|S|}$. Let $H_S\subset \RR^{|S|}$ denote the hyperplane perpendicular to $\alpha_S$. Then our construction depends on being able to choose a vector in the complement of $H_S$ for all subsets $S$.  But the cardinality of $S$ is changing as we range over all possibilities. To overcome this we will work with the larger space $\RR^m=\RR^{|S|}\times\RR^{|S^c|}$. There are finitely many choices of $S$ therefore $\bigcup_{S\subset [m]} H_S\times\RR^{|S^c|}\neq \RR^m$. Then for  
$v\in \left(\bigcup_{S\subset [m]} H_S\times\RR^{|S^c|}\right)^c$ we defined 
\[ \hat{x}_i(j)=\begin{cases} 
      x_i(j) & j\in [n] \\
      v(i) & j=n+1. 
   \end{cases}\]
Then it follows that $\hat{X}=\{\hat{x}_i\}_{i=1}^m$ has the complement property in $\RR^{n+1}$.
\vskip10pt
For necessity assume $X$ does phase retrieval but does not have the overcomplete complement property. Any spanning set that is a basis cannot be 1-lifted since there will not be enough vectors to span $\RR^{n+1}$.
\end{proof}

The result above may be generalized for a $k$-lift with minimal effort. Naturally the overcompleteness of each subset $S$ is critical in determining what integers $k$ are plausible. More specifically, we define the lifting number of a phase retrievable frame as follows:

\begin{definition} Given a frame $X = \{x_i\}_{i\in [m]}\subset\RR^n$,  let 
\[L_{X}=\min\{|S|-n: \lspan \{x_i\}_{i\in S}=\RR^n\text{  and  } |S|\geq |S^c|\},\]
then $L_X$ is the lifting number for the frame $X$.
\end{definition}

From the previous theorem we see immediately that if $\Phi$ has the overcomplete complement property then $L_X\geq 1$. The lifting number tells us how many dimensions higher we can lift $\Phi$. If $L_X\geq 1$ then when we 1-lift, each overcomplete spanning subset will be lifted to a spanning set in $\RR^{n+1}$ with the same cardinality. If $L_X >1$ that means each spanning subset with higher cardinality ($S$ or $S^c$) will be lifted to a spanning set which is still not a basis in $\RR^{n+1}$, hence can be lifted again. The idea is that after each lift, the lifting number of the subsequent lifted frame $\hat{\Phi}$ is one smaller than $L_X$. That is, if $\hat{X}$ is a 1-lift of $X$ then $L_{\hat{X}}=L_X -1$.

\begin{corollary}
$X\subset \RR^n$ can be $k$-lifted if and only if $k\leq L_X$.
\end{corollary}

\begin{theorem}
If a frame $X\subset\RR^n$ contains $2n + 2m +1$ vectors with a $2n +2m$ full spark subset. $X$ can be $(m+1)$-lifted.
\end{theorem}

\begin{proof}
Clearly if a frame contains a $2n + 2m$ full spark subset it does phase retrieval as it contains a $2n-1$ full spark subset which already does phase retrieval. Let $X = \{x_i\}_{i\in [2n+2m+1]}$ and $H = \{x_i\}_{i\in [2n+2m]}$ be a full spark subset. Given any $S\subset [2n+2m]$, if $S$ contains more than half the elements in $H$ then it will be a spanning set with more than $n+m$ vectors hence its cardinality minus $n$ will be greater than $m$ hence at least $m+1$. If it contains less than half of the elements of $H$ then the same holds for $S^c$. If it contained exactly half then both $S$ and $S^c$ will contain $n+m$ elements of $H$ hence they both span. Whichever set that contains $x_{2n+2m+1}$ will be a spanning set of cardinality $n+m+1$. Hence $X$ will have lifting number $m+1$.
\end{proof}
The set of $2n + 2m +1$ full spark vectors is open, dense, and contains a subset of $2n+2m$ full spark vectors. Then the previous theorem shows the set of $2n + 2m +1$ vectors in $\RR^n$ that can be $m+1$-lifted contains an open dense set. Hence ``almost every" set of $2n + 2m +1$ or $2n + 2m +2$ vectors can be $m+1$-lifted.

The infinite dimensional version of this looks like the following.
Note that this is not a classification of the liftable phase
retrieving frames but a sufficient condition.
\begin{remark} In $\ell_2$, by a lift we ``add'' a coordinate at the beginning of the vector. That is, if $\hat{x}$ is a lift of $x$ then $P\hat{x}=(0,x(1), x(2), \ldots),$ where $P$ is the orthogonal projection onto $e_1^\perp$. 
\end{remark}

\begin{theorem}
Let $X=\{x_i\}_{i=1}^{\infty}$ be a frame for $\ell_2$ doing phase
retrieval and let $Y=\{y_i\}_{i=1}^{\infty}$ be a linearly dependent
spanning set in $\ell_2$.  Then $X \cup Y$ can be lifted to a
phase retrieving frame for $\ell_2$.
\end{theorem}

\begin{proof}

Let $X=\{x_i\}_{i=1}^\infty$  and $Y=\{y_i\}_{i=1}^\infty$ be as in the theorem. We show that we can lift this union to one higher dimension and maintain phase retrieval.
Let $L$ be the right shift operator on $\ell_2$, i.e. if $x=(x(1), x(2), \ldots)$, then $Lx=(0, x(1), x(2), \ldots)$. Replace vectors in $X$ by $\hat{X}=\{\hat{x}_i\}_{i=1}^\infty$ where $\hat{x}_i=Lx_i$. The idea for $\hat{Y}=\{\hat{y}_i\}_{i=1}^\infty$ is very similar to the proof in Theorem \ref{FiniteLift}. We show existence of a vector $v=(v(1), v(2), \ldots)\in\ell_2$ such that 
$\hat{y}_i=v(i)e_1 +Ly_i$ will have the desired property to assure $\hat{X}\cup\hat{Y}$ does phase retrieval. 

Since
$\{y_i\}_{i=1}^\infty$ is linearly 
dependent, there exists a sequence of scalars $\alpha=\{\alpha_i\}_{i=1}^\infty$ with all but a finite number equal to zero,
such that $\sum_{i=1}^{\infty}\alpha_iy_i = 0$. Denote $H_i=e_i^\perp\subset \ell_2$ and note that by the Baire Category Theorem 
\[\left[(\cup_{i=1}^\infty H_i)\cup \alpha^\perp\right]^c\neq\emptyset.\]

Let $v\in\left[(\cup_{i=1}^\infty H_i)\cup \alpha^\perp\right]^c$ and define $\hat{y}_i$ as stated above. Note that $v$ has all non-zero coordinates and $\langle v, \alpha\rangle\not=0$. Moreover,
\[\sum_{i=1}^{\infty}\alpha_i\hat{y}_i=\sum_{i=1}^{\infty}(\alpha_i(v(i)e_1+Ly_i)=\langle \alpha, v\rangle e_1.\] 
Let any $j\geq 1$ and $\epsilon>0$. Since $\{y_i\}_{i=1}^\infty$ spans $\ell_2$, there is a finite subset $I_j$ and scalars $\{\beta_k\}_{k\in I_j}$ such that
\[\|e_j-\sum_{k\in I_j}\beta_ky_k\|<\epsilon.\]
This implies
\[\|e_{j+1}-\sum_{k\in I_j}\beta_k(\hat{y}_k-v(k)e_1)\|<\epsilon, \mbox{ for all }j\geq 1.\] Since $e_1\in \lspan\{\hat{y}_i\}_{i=1}^\infty$, $e_j\in \overline{\lspan}\{\hat{y}_i\}_{i=1}^\infty$ for all $j$, $\hat{Y}=\{\hat{y}_i\}_{i=1}^\infty$ spans $\ell_2$.

Now we will show that $\hat{X}\cup \hat{Y}$ satisfies Edidin's theorem. Since $\langle e_1, \hat{y}_i\rangle=v(i)\not=0$, the projection of $e_1$ on the vectors of $\hat{Y}$ span $\ell_2$. Let any non-zero vector $x\not= e_1$, the projection of $x$ onto the $\hat{x}_i$'s will spans $e_1^\perp\subset \ell_2$. Note that $x$ cannot be orthogonal to all $\hat{y}_i$ since these vectors span $\ell_2$. Let $\hat{y_j}$ be one such vector.  Since $\hat{y}_j$ is outside of $e_1^\perp\subset \ell_2$, the projection of $x$ onto the vectors of $\hat{X}\cup \hat{Y}$ span $\ell_2$ as well. Hence $\hat{X}\cup\hat{Y}$ does phase retrieval.
\end{proof}

\section{Finite Dimensional Results which are not Known in
Infinite Dimensions}

It is known \cite{CCJ} that there are two orthonormal bases for
$\RR^n$ which do phase retrieval.  We do not know the same
for $\ell_2$.

\begin{problem}
\item Are there two Riesz bases for $\ell_2$ which do phase retrieval?
\end{problem}

However, we can do phase retrieval with three Riesz sequences
for $\ell_2$.

\begin{proposition}
There are three Riesz sequences for $\ell_2$ which do phase
retrieval.
\end{proposition}

\begin{proof}
For every $n$ let $H_n=\lspan \{e_i\}_{i=1}^{3^{n+1}}$. Choose $\{u_{nij}\}_{i=3^{n}+1,j=1}^{3^{n+1}\;\;\;\;,3}$ so that they are full spark in $H_n$ and 
\[\sum_{i=3^n+1}^{3^{n+1}}\|u_{nij}-e_i\|^2\leq \frac{1}{2^{n+1}}\;\;\;\; \text{for}\;\;\; j=1,2,3\]
Then the collection $\{u_{nij}: 3^{n}+1\leq i\leq 3^{n+1}, 1\leq j\leq 3, 1\leq n<\infty\}$ does phase retrieval. Moreover, 
\[\sum_{n=1}^\infty \sum_{i=3^n+1}^{3^{n+1}}\|u_{nij}-e_i\|^2\leq \frac{1}{2}\;\;\;\; \text{for}\;\;\; j=1,2,3,\]
therefore $\{u_{nij}\}_{n=1, i=3^{n+1}}^{\infty, \;\;\; 3^{n+1}}$ is a Riesz sequence for $j=1,2,3.$
\end{proof}

\section{Sets Which do Phase Retrieval in $\ell_2$}

\begin{theorem}
Assume we have subspaces $W_1 \subset W_2 \subset \cdots\subset \ell_2$
and vectors $\{x_{ij}\}_{j\in I_i}$ doing phase
retrieval in $W_i$ for every i.  Finally, assume $\cup_{i=1}^{\infty}W_i$ is dense in $\ell_2.$
Then $\{x_{ij}\}_{i=1,j\in I_i}^{\infty}$ does phase retrieval
in $\ell_2$.
\end{theorem}

\begin{proof}
We will check the complement property. Observe that a partition of vectors $\{x_{ij}\}_{i=1,j\in I_i}^{\infty}$ induces a partition for vectors $\{x_{ij}\}_{j\in I_i}\subset W_i$. By assumption $\{x_{ij}\}_{j\in I_i}$ does phase
retrieval on $W_i$, therefore for each $i=1,2,\ldots$ 
\[ \mbox{ either }
W_i\subset \overline{\lspan}\{x_{ij}\}_{(i,j)\in I} \mbox{ or }
W_i \subset \overline{\lspan}\{x_{ij}\}_{(i,j)\in I^c} .\]
Then either $I$ or $I^c$ contains infinitely many $W_i$, without loss of generality we assume it is $I$. This means that for
infinitely many $i$, 
\[ W_i \subset \overline{\lspan}\{x_{ij}\}_{(i,j)\in I}.\]
Since $W_i \subset W_{i+1}$ for all $i$, 
\[ \cup_{i=1}^{\infty}W_i \subset \overline{\lspan}\{x_{ij}\}_{(i,j)\in I},\]
and so the closure of the right hand set is $\ell_2$.  This shows our family of vectors have complement property and hence do
phase retrieval on $\ell_2$.
\end{proof}

\begin{theorem}
Let $P_n$ be the orthogonal projection of $\ell_2$ onto $E_n=\lspan\{e_i\}_{i=1}^n$.
There is a set of vectors $Y=\{y_{ni}\}_{n=1,i=1}^{\infty,\;\;\;\infty}$ that does not
do phase retrieval on $\ell_2$, but  $X=\{x_{ni}\}_{n=1,i=1}^{\infty,\;\;\;\infty}=\{P_n{y_{ni}}\}_{n=1,i=1 }^{\infty,\;\;\; \infty}$ does phase retrieval
in $\ell_2$. Moreover, finite subsets of $X$ do phase retrieval on $E_n$ for every $n$.
\end{theorem}
\begin{proof}
For each $n\in\NN$, let $X_n$ be a finite set of vectors $\{x_{ni}\}_{i\in I_n}$ contained in $E_n$ that does phase retrieval in $E_n$. For example consider a full spark set in $E_n$ embedded in $\ell_2$ by adding zero to all other entries. We know that $X=\{x_{ni}\}_{n=1, i\in I_n}^\infty$ does phase retrieval in $\ell_2$. It is sufficient to show that for each $n$ and $i$, there exists $y_{ni}$, with $P_ny_{ni}=x_{ni}$, such that the $y_{ni}$ is contained in a fixed hyperplane for all $n, i$. Let $w$ be the vector with infinitely many non-zero coordinates. For each $n$, $x_{ni}$ has finite support contained in the first $n$ coordinates, for all $i\in I_n$. Then there is $j>n$ such that $w(j)\not=0$. Define $y_{ni}=x_{ni}-\dfrac{\langle x_{ni},w\rangle}{w(j)}e_j$, for $i\in I_n$. It follows that $\langle y_{ni}, w\rangle=0$, and hence $y_{ni}\subset w^\perp$ for all $n, i$. This completes the proof. 
\end{proof}

In the following, we will show how to create a new phase retrieval set by translating the vectors of the original one in the same direction. First, we will need a lemma.

\begin{lemma}\label{Lem2}
If $\{x_i\}_{i=1}^{\infty}$ is Bessel in $\ell_2$, then for
every $v\in \ell_2$, 
\[ \lim_{i\rightarrow \infty}\langle v,x_i\rangle =0.\]
\end{lemma}

\begin{proof}
Given a vector $v$, we have
\[ \sum_{i=1}^{\infty}|\langle v,x_i\rangle|^2< \infty,\]
hence $\lim_{i\rightarrow \infty}|\langle v,x_i\rangle|=0.$
\end{proof}

\begin{remark}
Note that if any $\{x_i\}_{i=1}^{\infty}$ does phase retrieval,
then
\[ \left \{ \frac{1}{\|x_i\|2^i}x_i\right \}_{i=1}^{\infty}\]
is Bessel and also does phase retrieval.
\end{remark}

\begin{theorem}\label{PRTranslateThm}
Assume $\{x_i\}_{i=1}^{\infty}$ is a Bessel sequence in $\ell_2$ and
does phase retrieval.  Then for every $v\in \ell_2$, 
$\{x_i+v\}_{i=1}^{\infty}$ does phase retrieval.
\end{theorem}

\begin{proof}
Assume 
\[ |\langle x,x_i+v\rangle|=|\langle y, x_i+v\rangle|,
\mbox{ for all } i=1,2,\ldots.\]
Let
\[ I=\{i:\langle x,x_i+v\rangle=\langle y, x_i+v\rangle\}.\]
Then either $|I|$ or $|I^c|$ is infinite.  By the complement property,
either $\{x_i\}_{i\in I}$ or $\{x_i\}_{i\in I^c}$ spans the space.
Without loss of generality,
assume $\{x_i\}_{i\in I}$ spans $\ell_2$.  Now,
\[ \langle x,x_i+v\rangle=\langle y, x_i+v\rangle,\mbox{ for all }
i\in I,\]
and so
\[ \langle x-y,x_i\rangle = \langle y-x,v\rangle,\mbox{ for all }
i\in I.\]
By Lemma \ref{Lem2}, 
\[ \langle y-x,v\rangle = 0 = \langle x-y,x_i\rangle, \mbox{ for all } i\in I.\]
It follows that $x-y=0$.
\end{proof}

\begin{remark}
Note that $\{x_i+v\}_{i=1}^{\infty}$ is not Bessel.  But we can
scale it to be Bessel and it still does phase retrieval.
\end{remark}

\begin{corollary}
We can perturb a family doing phase retrieval in $\ell_2$ as long
as we perturb all vectors in the {\it same direction}.
\end{corollary}

\begin{corollary}
Given $\{x_i\}_{i=1}^{\infty}$ which is Bessel and does phase retrieval in $\ell_2$, for any $x_j$ the family $\{x_i-x_j\}_{i=1}^\infty$ does phase retrieval.
\end{corollary}
\begin{proof}  Let $v=-x_{j}$ and note that it follows that $\{x_i+v\}_{i=1}^\infty$ does phase retrieval by Theorem \ref{PRTranslateThm}.
\end{proof}

Note repeating the argument in the previous corollary, it is possible to ``delete" a finite number of vectors by translating the system and scaling the set so they are Bessel.

\begin{proposition}
There is a family of vectors in $\ell_2$
doing phase retrieval where each of the vectors has all non-zero
coordinates with respect to the unit vectors.
\end{proposition}

\begin{proof}
Let $\{x_i\}_{i=1}^{\infty}$ do phase retrieval.  Let $\{e_i\}_{i=1}^{\infty}$ be the unit vectors.  For any $j=1,2,\ldots$ the family
$\{x_i(j)\}_{i=1}^{\infty}$ is a countable set of real numbers
so choose a real number $a_j\not=- w_i(j)$ and $0<a_j< \frac{1}{2^j}$ for all $i=1,2,\ldots$. 
Let $v=(a_1,a_2,\ldots)$.  Then $\{x_i+v\}_{i=1}^{\infty}$ does
phase retrieval and each vector has all non-zero coordinates.
\end{proof}

\end{document}